\documentclass{amsart}
\usepackage{graphicx}
\usepackage{amsmath}
\usepackage{amsthm}
\usepackage{amsfonts}
\usepackage{amssymb, amscd}

\usepackage{color}

\usepackage[all]{xy}

\newtheorem{theorem}{Theorem}[section]

\newtheorem{proposition}[theorem]{Proposition}
\newtheorem{example}[theorem]{Example}
\theoremstyle{definition}
\newtheorem{definition}[theorem]{Definition}
\theoremstyle{remark}
\newtheorem{remark}[theorem]{Remark}
\numberwithin{equation}{section}

\newcommand{\K}{\mathbb K}

\newcommand{\A}{\mathcal{A}}

\newcommand{\g}{\mathfrak{g}}

\newcommand{\sll}{\mathfrak{sl}_2(\mathbb{K})}

\begin{document}

\title[  Hom-dendriform algebras and Rota-Baxter Hom-algebras]
{  Hom-dendriform algebras and Rota-Baxter Hom-algebras}
\author{Abdenacer MAKHLOUF  }
\address{Abdenacer Makhlouf, Universit\'{e} de Haute Alsace,  Laboratoire de Math\'{e}matiques, Informatique et Applications,
4, rue des Fr\`{e}res Lumi\`{e}re F-68093 Mulhouse, France}%
\email{Abdenacer.Makhlouf@uha.fr}

\thanks {
}

%
\begin{abstract}
The aim of this paper is to introduce and study  Rota-Baxter Hom-algebras. Moreover we introduce a generalization of the dendriform algebras and tridendriform algebras  by twisting the identities by mean of a linear map. Then we explore the connections between these categories of Hom-algebras.
\end{abstract}
\maketitle

\section*{Introduction}
The study of nonassociative algebras was originally motivated by
certain problems in physics and other branches of mathematics. The Hom-algebra structures
arose first in quasi-deformation of Lie algebras of vector fields.
Discrete modifications of vector fields via twisted derivations lead
to Hom-Lie and quasi-Hom-Lie structures in which the Jacobi
condition is twisted. The first examples of $q$-deformations, in which the derivations are replaced by $\sigma$-derivations,
concerned the Witt and Virasoro algebras, see for example
\cite{AizawaSaito,ChaiElinPop,ChaiKuLukPopPresn,ChaiIsKuLuk,ChaiPopPres,CurtrZachos1,
DaskaloyannisGendefVir,Kassel1, LiuKeQin,Hu}. A  general study and
construction of Hom-Lie algebras are considered in
\cite{HLS,LS1,LS2} and a more general framework bordering color and
super Lie algebras was introduced
 in \cite{HLS,LS1,LS2,LS3}. In the subclass of Hom-Lie
algebras skew-symmetry is untwisted, whereas the Jacobi identity is
twisted by a single linear map and contains three terms as in Lie
algebras, reducing to ordinary Lie algebras when the twisting linear
map is the identity map.

The notion of Hom-associative algebras generalizing associative
algebras to a situation where associativity law is twisted by a
linear map was introduced  in \cite{MS}, it turns out  that the
commutator bracket multiplication defined using the multiplication
in a Hom-associative algebra leads naturally to Hom-Lie algebras.
This provided a different  way of constructing Hom-Lie algebras. Also in \cite{MS}, the
Hom-Lie-admissible algebras and more general $G$-Hom-associative
algebras with subclasses of Hom-Vinberg and Hom-preLie algebras,
generalizing to the twisted situation Lie-admissible algebras,
$G$-associative algebras, Vinberg and preLie algebras respectively are introduced and it is shown that for these classes of algebras the operation of taking
commutator leads to Hom-Lie algebras as well. The enveloping
algebras of Hom-Lie algebras were discussed in \cite{Yau:EnvLieAlg}.
The fundamentals of the formal deformation theory and associated
cohomology structures for Hom-Lie algebras have been considered
initially  in \cite{HomDeform} and completed in \cite{AEM}. Simultaneously, in \cite{Yau:HomolHomLie}
elements of homology for Hom-Lie algebras  have been developed.
In \cite{HomHopf} and \cite{HomAlgHomCoalg}, the theory of
Hom-coalgebras and related structures are developed. Further development could be found in \cite{Mak-ElHamd:defoHom-alter,Mak:HomAlterna2010,Mak:Almeria,AmmarMakhloufJA2010,AM2008,Mak:HomAlterna2010,Canepl2009,JinLi,Yau:comodule}.

Dendriform algebras were introduced by Loday in \cite{Loday1}. Dendriform algebras are algebras with two operations, which dichotomize the notion of associative algebra.  The motivation to introduce these algebraic structures with two generating operations comes from $K$-theory. It turned out later that they are connected to several areas in mathematics and physics, including  Hopf algebras, homotopy Gerstenhaber algebra, operads, homology, combinatorics 	and quantum field theory where they  occur  in the theory of renormalization of Connes and Kreimer. Later
the notion of tridendriform algebra were introduced by Loday and Ronco in their study of polytopes and Koszul duality, see \cite{Loday-Ronco04}. A tridendriform algebra is a vector space equipped with 3 binary operations  satisfying seven relations.

The Rota-Baxter operator has appeared in a wide range of areas in pure and applied mathematics. The paradigmatic example of Rota-Baxter operator concerns the integration by parts formula of continuous functions. The algebraic formulation of Rota-Baxter algebra appeared first in G. Baxter's works in probability study of fluctuation theory. This algebra was intensively studied by G.C. Rota in connection with combinatorics. Another connection with Rota-Baxter algebra with mathematical physics was found by A. Connes and D. Kreimer in their Hopf algebra approach to renormalization  of quantum field theory. This seminal work gives rise to an important development including Rota-Baxter algebras and their connections to other algebraic structure (see \cite{Aguiar0,Aguiar,KEF1,KEF-Guo1,KEF-Manchon_JA2009,KEF-Manchon2009,KEF-Bondia-Patras,KEF-Machon-Patras,Guo1,Guo2,Guo3,Guo4,Li-Hou-Bai07}).

The purpose of this paper is  to study Rota-Baxter Hom-algebras. We introduce Hom-dendriform and Hom-tridendriform algebras and then explore the connections between all these categories of Hom-algebras. We summarize in the first section the basis of Hom-algebras and recall the definitions and some properties of Hom-associative, Hom-Lie and Hom-preLie algebras. In Section 2, we introduce the notions of Hom-dendriform algebras and Hom-tridendriform algebras and provide  constructions of these algebras and their relationships with Hom-preLie algebras. Section 3 is dedicated to Hom-associative Rota-Baxter algebras, we extend the classical notion of associative Rota-Baxter algebra and show some  constructions. In Section 4 we establish functors between the category of Hom-associative Rota-Baxter algebras and the categories of Hom-preLie, Hom-dendriform and Hom-tridendriform algebras. In Section 5 we discuss the Rota-Baxter operator in the context of Hom-nonassociative algebras mainly for Hom-Lie algebras.

\section{ Hom-associative, Hom-Lie   and Hom-preLie algebras} \label{sect1}

In this section  we
summarize the definitions and some properties of Hom-associative, Hom-Lie
 and Hom-preLie algebraic structures  (see \cite{MS}) generalizing
the well known  associative, Lie and preLie  algebras by twisting the identities with a linear map.

Throughout the article we let  $\mathbb{K}$ be an algebraically
closed field of characteristic $0$.  We mean by a Hom-algebra a triple $(A,\mu,\alpha )$ consisting of a vector space $A$ on which $\mu : A\times A \rightarrow A$ is a bilinear map (or  $\mu : A\otimes A \rightarrow A$ is a linear map) and  a $\alpha :A \rightarrow A$ is linear map.
A Hom-algebra $(A, \mu, \alpha)$  is said to be
\emph{multiplicative} if  $\forall x,y\in A$ we have $\alpha([x,y])=[\alpha (x),\alpha (y)]$.

Let $\left( A,\mu,\alpha \right) $ and
$A^{\prime }=\left( A^{\prime },\mu ^{\prime
},\alpha^{\prime }\right) $ be two Hom-algebras of a given type. A linear map
$f\ :A\rightarrow A^{\prime }$ is
a \emph{morphism of Hom-algebras} if%
$$
\mu ^{\prime }\circ (f\otimes f)=f\circ \mu \quad \text{
and } \qquad f\circ \alpha=\alpha^{\prime }\circ f.
$$

In particular, Hom-algebras $\left( A,\mu,\alpha \right) $ and
$\left( A,\mu ^{\prime },\alpha^{\prime }\right) $ are isomorphic if
there exists a
bijective linear map $f\ $such that%
$
\mu =f^{-1}\circ \mu ^{\prime }\circ (f\otimes f)$
and $\alpha= f^{-1}\circ \alpha^{\prime }\circ
f.
$

A subspace  $H$ of $A$ is said to be a \emph{subalgebra} if for $x,y\in H$  we have $\mu (x,y)\in H$ and $\alpha (x)\in H.$ A subspace  $I$ of $A$ is said to be an \emph{\emph{ideal}} if for $x\in I$ and $y\in A$ we have $\mu (x,y)\in I$ and $\alpha (x)\in I.$

In all the examples involving  the unspecified products are either given by skewsymmetry or equal to zero.

\subsection{Hom-associative algebras}
The Hom-associative
algebras were  introduced
by the author and Silvestrov in \cite{MS}.
\begin{definition}[Hom-associative algebra]
A \emph{Hom-associative algebra}  is a triple $( A, \cdot,
\alpha) $ consisting of a vector space  $A$ on which   $\cdot :A\otimes A\rightarrow  A$
and $\alpha: A \rightarrow A$ are  linear maps, satisfying
\begin{equation}\label{Hom-ass}
\alpha(x)\cdot (y\cdot z)= (x\cdot y)\cdot \alpha (z).
\end{equation}
\end{definition}

\begin{example}\label{example1ass}
Let $\{x_1,x_2,x_3\}$  be a basis of a $3$-dimensional linear space
$A$ over $\K$. The following multiplication $\cdot$ and linear map
$\alpha$ on $A$ define Hom-associative algebras over $\K^3${\rm :}
$$
\begin{array}{ll}
\begin{array}{lll}
  x_1\cdot x_1&=& a\ x_1, \ \\
 x_1\cdot  x_2&=& x_2 \cdot  x_1=a\ x_2,\\
x_1 \cdot  x_3 &=& x_3\cdot  x_1=b\ x_3,\\
 \end{array}
 & \quad
 \begin{array}{lll}
x_2\cdot  x_2 &=& a\ x_2, \ \\
 x_2\cdot   x_3&=& b\ x_3, \ \\
 x_3\cdot  x_2&=&  x_3\cdot x_3=0,
  \end{array}
\end{array}
$$

$$  \alpha (x_1)= a\ x_1, \quad
 \alpha (x_2) =a\ x_2 , \quad
   \alpha (x_3)=b\ x_3,
$$
where $a,b$ are parameters in $\K$. The algebras are not associative
when $a\neq b$ and $b\neq 0$, since
$$ (x_1\cdot  x_1)\cdot x_3-  x_1\cdot
(x_1\cdot x_3)=(a-b)b x_3.$$
\end{example}

\begin{example}[Polynomial Hom-associative algebra \cite{Yau:homology}]
Consider the polynomial algebra $\A=\K [x_1,\cdots x_n]$ in $n$
variables. Let $\alpha$ be an algebra endomorphism of $\A$ which is
uniquely determined by the $n$ polynomials $\alpha (x_i)=\sum
{\lambda_{i;r_1,\cdots,r_n}x^{r_1}_1,\cdots x^{r_n}_n}$ for $1\leq i
\leq n$. Define $\mu$ by
\begin{equation}
\mu (f,g)=f(\alpha (x_1),\cdots \alpha (x_n))g(\alpha (x_1),\cdots
\alpha (x_n))
\end{equation}
for $f,g$ in $\A$. Then,  $(\A ,\mu,\alpha)$ is a Hom-associative algebra.

\end{example}
\begin{example}[\cite{Yau:comodule}]
Let $\A=(A,\mu,\alpha )$ be a Hom-associative algebra. Then
$(\mathcal{M}_n(\A),\mu',\alpha' )$, where $\mathcal{M}_n(\A)$ is
the vector space of $n\times n$ matrix with entries in $A$, is also
a Hom-associative algebra in which the multiplication $\mu'$ is
given by matrix multiplication and $\mu$,  and $\alpha'$ is given by
$\alpha$ in each entry.

\end{example}

\subsection{Hom-Lie algebras}The  notion of Hom-Lie algebra  was introduced by Hartwig, Larsson and
Silvestrov in \cite{HLS,LS1,LS2} motivated initially by examples of deformed
Lie algebras coming from twisted discretizations of vector fields.
In this article, we follow notations and a slightly more general definition
of Hom-Lie algebras from \cite{MS}.
\begin{definition}[Hom-Lie algebra] \label{def:HomLie}
A \emph{Hom-Lie algebra} is a triple $(\g, [\ ,
\ ], \alpha)$ consisting of a vector space $\g$ on which  $[\ , \ ]: \g\times \g \rightarrow \g$ is
a bilinear map and $\alpha: \g \rightarrow \g$
 a linear map
 satisfying
\begin{eqnarray} & [x,y]=-[y,x],
\quad {\text{(skew-symmetry)}} \\ \label{HomJacobiCondition} &
\circlearrowleft_{x,y,z}{[\alpha(x),[y,z]]}=0 \quad
{\text{(Hom-Jacobi condition)}}
\end{eqnarray}
for all $x, y, z$ in $\g$, where $\circlearrowleft_{x,y,z}$ denotes
summation over the cyclic permutation on $x,y,z$.
\end{definition}
We recover classical Lie algebras when $\alpha =id_\g$ and the identity \eqref{HomJacobiCondition} is the Jacobi identity in this case.

\begin{example}
Let $\{x_1,x_2,x_3\}$  be a basis of a $3$-dimensional vector space
$\g$ over $\K$. The following bracket and   linear map $\alpha$ on
$\g=\K^3$ define a Hom-Lie algebra over $\K${\rm :}
$$
\begin{array}{cc}
\begin{array}{ccc}
 [ x_1, x_2 ] &= &a x_1 +b x_3 \\ {}
 [x_1, x_3 ]&=& c x_2  \\ {}
 [ x_2,x_3 ] & = & d x_1+2 a x_3,
 \end{array}
 & \quad
  \begin{array}{ccc}
  \alpha (x_1)&=&x_1 \\
 \alpha (x_2)&=&2 x_2 \\
   \alpha (x_3)&=&2 x_3
  \end{array}
\end{array}
$$
with $[ x_2, x_1 ]$, $[x_3, x_1 ]$ and  $[
x_3,x_2 ]$ defined via skewsymmetry. It is not a
Lie algebra if $a\neq0$ and $c\neq0$,
since
$$[x_1,[x_2,x_3]]+[x_3,[x_1,x_2]]
+[x_2,[x_3,x_1]]= a c x_2.$$
\end{example}

\begin{example}[Jackson $\mathfrak{sl}_2$]
The Jackson $\mathfrak{sl}_2$ is a $q$-deformation of the classical  $\mathfrak{sl}_2$. This family of Hom-Lie algebras was constructed in \cite{LS3}
using  a quasi-deformation scheme based on discretizing by means of
Jackson $q$-derivations a representation of $\sll$ by
one-dimensional vector fields (first order ordinary differential
operators) and using the twisted commutator bracket defined in
\cite{HLS}. It carries a  Hom-Lie algebra structure but not a Lie algebra structure. It is defined with respect to a basis $\{x_1,x_2,x_3\}$  by the brackets and a linear  map $\alpha$ such that
$$
\begin{array}{cc}
\begin{array}{ccc}
 [ x_1, x_2 ] &= &-2q x_2 \\ {}
 [x_1, x_3 ]&=& 2 x_3 \\ {}
 [ x_2,x_3 ] & = & - \frac{1}{2}(1+q)x_1,
 \end{array}
 & \quad
  \begin{array}{ccc}
  \alpha (x_1)&=&q x_1 \\
 \alpha (x_2)&=&q^2  x_2 \\
   \alpha (x_3)&=&q x_3
  \end{array}
\end{array}
$$
where $q$ is a parameter in $\K$. if $q=1$ we recover the classical $\mathfrak{sl}_2$.
\end{example}


There is a functor from the category of Hom-associative algebras in
the category of Hom-Lie algebras. It provides a different way for constructing Hom-Lie algebras by extending the fundamental construction of Lie algebras by associative algebras via commutator bracket.
\begin{proposition}[\cite{MS}]
Let $( A, \cdot, \alpha) $ be a Hom-associative algebra defined on the vector space $A$  by the
multiplication $\cdot$ and a homomorphism $\alpha$. Then the triple $( A, [~,~], \alpha) $, where the bracket
is defined for  $x,y \in A$ by  $ [ x,y ]=x\cdot  y-y\cdot x
$, is a Hom-Lie algebra.
\end{proposition}

\subsection{Hom-preLie algebras}
The Hom-preLie algebras were introduced in \cite{MS} in the study of Hom-Lie admissible algebras.
\begin{definition}[Hom-preLie algebras,]
A left Hom-preLie  algebra (resp. right Hom-preLie  algebra) is a triple $(A, \cdot, \alpha)$ consisting of
a vector space $A$, a bilinear map $\cdot: A\times A \rightarrow A$
and a homomorphism $\alpha$ satisfying
\begin{equation}
\alpha(x)\cdot(y,z)-(x\cdot y)\cdot z=\alpha(y)\cdot  (x\cdot z)-
(y\cdot x)\cdot  \alpha (z),
\end{equation}
resp.
\begin{equation}
\alpha(x)\cdot  (y,z)-
(x\cdot y)\cdot \alpha (z)=\alpha(x)\cdot  (z\cdot y)- (x\cdot  z)\cdot \alpha (y).
\end{equation}
\end{definition}

\begin{remark}Any Hom-associative algebra is a Hom-preLie algebras.

A left Hom-preLie algebra is the opposite algebra of the right Hom-preLie algebra. Both left and right Hom-preLie algebras are Hom-Lie-admissible algebras, that is the commutators define Hom-Lie algebras, see \cite{MS}.
\end{remark}

\section{Hom-dendriform algebras and Hom-Tridendriform algebras}
In this section, we introduce the notions of Hom-dendriform algebras and Hom-tridendriform algebras generalizing the classical dendriform and tridendriform algebras to Hom-algebras setting.
\subsection{Hom-dendriform algebras}
Dendriform algebras were introduced by Loday in \cite{Loday1}. Dendriform algebras are algebras with two operations, which dichotomize the notion of associative algebra.   We generalize now this  notion  by twisting the identities by a linear map.

\begin{definition}[Hom-dendriform algebra] \label{def:HomDendr}
A \emph{Hom-dendriform algebra} is a quadruple $(A, \prec,\succ, \alpha)$ consisting of a vector  space $A$ on which the operations $\prec, \succ : A\otimes A \rightarrow \g$
and $\alpha: A \rightarrow A$
 are linear maps
 satisfying
\begin{eqnarray}\label{HomDendriCondition1}  (x\prec y)\prec \alpha (z)&=& \alpha(x)\prec(y\prec z +y\succ z),
 \\ \label{HomDendriCondition2} (x\succ y)\prec \alpha (z)&=&\alpha(x)\succ(y\prec z),\\
\label{HomDendriCondition3} \alpha(x)\succ(y\succ z)&=&(x\prec y+x\succ y)\succ \alpha (z).
\end{eqnarray}
for  $x, y, z$ in $A$.
\end{definition}
We recover classical dendriform algebra when $\alpha =id$.

Let $(A, \prec,\succ, \alpha) $ and
$(A', \prec',\succ', \alpha') $ be two Hom-dendriform algebras. A linear map
$f\ :A\rightarrow A'$ is
a \emph{ Hom-dendriform algebras morphism} if%
$$
 \prec'\circ (f\otimes f)=f\circ  \prec,  \quad
   \succ'\circ (f\otimes f)=f\circ  \succ \quad \text{
and } \qquad f\circ \alpha=\alpha^{\prime }\circ f.
$$

We show now that we may construct  Hom-dendriform algebras starting from a classical dendriform algebra and an algebra endomorphism. We extend then the construction by composition  introduced by Yau in \cite{Yau:HomolHomLie}  for Lie and associative algebras.
\begin{theorem}\label{thmConstrHomDend}
Let $(A,\prec, \succ)$ be a dendriform  algebra and $\alpha :
A\rightarrow A$ be an dendriform algebra endomorphism. Then
$A_\alpha =(A,\prec_\alpha, \succ_\alpha,\alpha)$, where $\prec_\alpha=\alpha\circ\prec$ and $\succ_\alpha=\alpha\circ\succ$, is a Hom-dendriform algebra.

Moreover, suppose that $(A',\prec', \succ')$ is another dendriform
algebra and  $\alpha ' : A'\rightarrow A'$ is a dendriform algebra
endomorphism. If $f:A\rightarrow A'$ is a dendriform algebra morphism that
satisfies $f\circ\alpha=\alpha '\circ f$ then
$$f:(A,\prec_\alpha, \succ_\alpha,\alpha)\longrightarrow (A',{\prec'} _\alpha, {\succ '} _\alpha ,\alpha ')
$$
is a morphism of Hom-dendriform algebras.
\end{theorem}
\begin{proof}
Observe that
\begin{align*}
(x\prec_\alpha y)\prec_\alpha \alpha (z) =\alpha^2 ((x\prec y)\prec z), \\
(x\prec_\alpha y)\succ_\alpha \alpha (z) =\alpha^2 ((x\prec y)\succ z), \\
(x\succ_\alpha y)\succ_\alpha \alpha (z) =\alpha^2 ((x\succ y)\succ z), \\
(x\succ_\alpha y)\prec_\alpha \alpha (z) =\alpha^2 ((x\succ y)\prec z). \\
\end{align*}
And similarly
\begin{align*}
 \alpha (x)\prec_\alpha (y\prec_\alpha z) =\alpha^2 (x\prec (y\prec z)), \\
\alpha(x)\prec_\alpha (y\succ_\alpha  z) =\alpha^2 (x\prec (y\succ z)), \\
\alpha(x)\succ_\alpha( y\succ_\alpha  z) =\alpha^2 (x\succ( y\succ z)), \\
\alpha(x)\succ_\alpha (y\prec_\alpha  z) =\alpha^2 (x\succ (y\prec z)). \\
\end{align*}
 Therefore the
identities  \eqref{HomDendriCondition1},\eqref{HomDendriCondition2},\eqref{HomDendriCondition3} follow
obviously from the  identities satisfied by $(A,\prec, \succ)$. The second assertion is proved similarly.
\end{proof}

In the classical case the commutative dendriform algebras are also called Zinbiel algebras (see \cite{Loday0,Loday1}). The left and right operations are further required to identify, $x\prec y=y \succ x$. We call  commutative Hom-dendriform algebras Hom-Zinbiel algebras.

\begin{definition}[Hom-Zinbiel algebra] \label{def:HomZinbiel}
A \emph{Hom-Zinbiel algebra} is a triple $(A,\circ , \alpha)$ consisting of a vector  space $A$ on which  $\circ: A\otimes A \rightarrow A$
and $\alpha: A \rightarrow A$
 are linear maps
 satisfying
\begin{eqnarray}\label{HomZinbCondition1}  & (x\circ y)\circ \alpha (z)= \alpha(x)\circ(y\circ z) + \alpha(x)\circ(z\circ y).
\end{eqnarray}
for  $x, y, z$ in $A$.
\end{definition}
\begin{remark}
One may construct Hom-Zinbiel algebra by composition method starting from a classical Zinbiel algebra $(A,\circ )$ and an algebra endomorphism $\alpha$ by considering $(A,\circ_\alpha , \alpha)$, where $x\circ_\alpha y=\alpha(x\circ y).$
\end{remark}

We show now that Hom-dendriform algebra structure dichotomize the Hom-associative structure and provide a connection to Hom-preLie algebras.

\begin{proposition}
Let $(A, \prec,\succ, \alpha) $   be a Hom-dendriform algebra. Let $\star: A\otimes A\rightarrow A$ be a linear map defined   for $x,y\in A$ by
\begin{equation}\label{DendToAss}
x \star y=x\prec y+ x \succ y.
\end{equation}
Then $(A,\star, \alpha)$ is a Hom-associative algebra.
\end{proposition}
\begin{proof}For $x,y\in A$ we have
\begin{align*}
\alpha (x)\star (y\star z)&= \alpha (x)\star (y\prec z+ y \succ z),\\
\ &= \alpha (x)\prec (y\prec z+ y \succ z)+\alpha (x)\succ (y\prec z+ y \succ z),\\
\ &= (x\prec y)\prec \alpha (z)+\alpha (x)\succ (y\prec z)+\alpha (x)\succ( y \succ z),\\
\ &= (x\prec y)\prec \alpha (z)+(x\succ y)\prec \alpha (z)+(x \prec y+x\succ y) \succ \alpha (z),\\
\ &= (x\prec y+x\succ y)\prec \alpha (z)+(x \prec y+x\succ y) \succ \alpha (z),\\
\ &= (x\star y)\prec \alpha (z)+(x\star y) \succ \alpha (z),\\
\ &= (x\star y)\star \alpha (z).
\end{align*}

\end{proof}

\begin{proposition}\label{HDendtoHpLie}
Let $(A, \prec,\succ, \alpha) $   be a Hom-dendriform algebra. Let $\lhd: A\otimes A\rightarrow A$ and $\rhd: A\otimes A\rightarrow A$ be  linear maps defined   for $x,y\in A$ by
\begin{equation}\label{DendToAss}
x\rhd y=x\succ y- y \prec x \quad\text{and}\quad  x\lhd y=x\prec y- y \succ x.
\end{equation}
Then $(A,\rhd, \alpha)$ is a left Hom-preLie algebra and $(A,\lhd, \alpha)$ is a right Hom-preLie algebra.
\end{proposition}
\begin{proof}For $x,y\in A$ we have
\begin{align*}
\alpha (x)\rhd (y\rhd z)&=\alpha (x)\rhd (y\succ z-z\prec y),\\
\ &= \alpha (x)\succ (y\succ z)-\alpha (x)\succ (z\prec y)-(y\succ z)\prec \alpha (x)+(z\prec y)\prec \alpha (x).
\end{align*}
and
\begin{align*}
(x\rhd y)\rhd \alpha (z)&=(x\succ y-y\prec x)\rhd \alpha (z),\\
\ &= (x\succ y)\succ \alpha (z)-(y\prec x)\succ \alpha (z)-\alpha (z)\prec (x\succ y)+\alpha (z)\prec (y\prec x).
\end{align*}
Using \eqref{HomDendriCondition1} and \eqref{HomDendriCondition3}, we may write
\begin{align*}
\alpha (x)\rhd (y\rhd z)&= (x\prec y)\succ \alpha (z)+(x\succ y)\succ \alpha (z)
-\alpha (x)\succ (z\prec y)\\
\ &-(y\succ z)\prec \alpha (x)+
 \alpha(z)\prec (y\prec x)+\alpha(z)\prec (y\succ x).
\end{align*}
Direct simplification and identity \eqref{HomDendriCondition1} lead to
\begin{equation*}
\alpha (x)\rhd (y\rhd z)-(x\rhd y)\rhd \alpha (z)-\alpha (y)\rhd (x\rhd z+(y\rhd x)\rhd \alpha (z)=0.
\end{equation*}
Similar proof shows the right Hom-preLie structure.
\end{proof}
\begin{remark}If $(A, \prec,\succ, \alpha) $   is  a commutative Hom-Dendriform algebra then the corresponding left and right Hom-preLie algebras vanish.
\end{remark}

\subsection{Hom-tridendriform algebras}
The notion of tridendriform algebra were introduced by Loday and Ronco in \cite{Loday-Ronco04}. A tridendriform algebra is a vector space equipped with 3 binary operations $<,>\cdot$ satisfying seven relations. We extend this notion to Hom situation as follows:
\begin{definition}[Hom-Tridendriform algebra] \label{def:HomTriDendr}
A \emph{Hom-tridendriform algebra} is a quintuple $(A, \prec,\succ,\cdot, \alpha)$ consisting of a vector  space $A$ on which the operations  $\prec, \succ ,\cdot: A\otimes A \rightarrow A$
and $\alpha: A \rightarrow A$
 are linear maps
 satisfying
\begin{eqnarray}\label{HomTriDendriCondition1}   (x\prec y)\prec \alpha (z)&=& \alpha(x)\prec(y\prec z +y\succ z+ y\cdot z),
 \\ \label{HomTridDendriCondition2} (x\succ y)\prec \alpha (z)&=&\alpha(x)\succ(y\prec z),\\
\label{HomTriDendriCondition3} \alpha(x)\succ(y\succ z)&=&(x\prec y+x\succ y+x\cdot y)\succ \alpha (z),\\
\label{HomTridDendriCondition4} (x\prec y)\cdot \alpha (z)&=&\alpha(x)\cdot(y\succ z),\\
\label{HomTridDendriCondition5} (x\succ y)\cdot \alpha (z)&=&\alpha(x)\succ(y\cdot z),\\
\label{HomTridDendriCondition6} (x\cdot y)\prec \alpha (z)&=&\alpha(x)\cdot(y\prec z),\\
\label{HomTridDendriCondition7} (x\cdot y)\cdot \alpha (z)&=&\alpha(x)\cdot(y\cdot z),\\
\end{eqnarray}
for  $x, y, z$ in $A$.
\end{definition}
We recover classical tridendriform algebra when $\alpha =id$.
\begin{remark}
Any Hom-tridendriform algebra gives a Hom-dendriform algebra by setting $x\cdot y=0$ for any $x,y\in A$.
\end{remark}

As in Theorem \ref{thmConstrHomDend}, Given a classical tridendriform algebra and an algebra endomorphism we may construct by composition a Hom-tridendriform algebra.
\begin{proposition}\label{thmConstrHomTriDend}
Let $(A,\prec, \succ,\cdot)$ be a tridendriform  algebra and $\alpha :
A\rightarrow A$ be an tridendriform algebra endomorphism. Then
$A_\alpha =(A,\prec_\alpha, \succ_\alpha,\cdot_\alpha,\alpha)$, where $\prec_\alpha=\alpha\circ\prec$,  $\succ_\alpha=\alpha\circ\succ$ and $\cdot_\alpha=\alpha\circ\cdot$, is a Hom-tridendriform algebra.

Moreover, suppose that $(A',\prec', \succ',\cdot')$ is another tridendriform
algebra and  $\alpha ' : A'\rightarrow A'$ is a tridendriform algebra
endomorphism. If $f:A\rightarrow A'$ is a tridendriform algebra morphism that
satisfies $f\circ\alpha=\alpha '\circ f$ then
$$f:(A,\prec_\alpha, \succ_\alpha,\cdot_\alpha,\alpha)\longrightarrow (A',{\prec'} _\alpha, {\succ '} _\alpha ,\cdot'_\alpha\alpha ')
$$
is a morphism of Hom-tridendriform algebras.
\end{proposition}

Similarly as in  \cite{KEF1}, we obtain in the Hom-algebras setting the following new operation:
\begin{proposition}\label{HTridendToHAss}
Let  $(A, \prec,\succ,\cdot, \alpha)$ be a Hom-tridendriform algebra and
  $\ast :A\otimes A \rightarrow A$ be an operation defined by $x\ast y=x\prec y+ x\succ y+x\cdot y$.
Then $(A,\ast,\alpha)$ is a Hom-associative algebra.
\end{proposition}
\begin{proof}Using the axioms of Hom-tridendriform algebras we have for $x,y\in A$
\begin{align*}
\alpha (x)\ast (y\ast z)&= \alpha (x)\ast (y\prec z+ y \succ z+y\cdot z),\\
\ &=\alpha (x)< (y\prec z+ y \succ z+y\cdot z) +\alpha (x)>(y\prec z+ y \succ z+y\cdot z)  \\
\ &+\alpha (x)\cdot (y\prec z+ y \succ z+y\cdot z), \\
\ &=(x<y)<\alpha (z)+(x>y)<\alpha (z)+(x<y+x>y+x\cdot y)>\alpha (z) \\
\ &+(x>y)\cdot \alpha (z)+(x\cdot y)<\alpha (z)+(x<y)\cdot \alpha (z)+(x\cdot y)\cdot \alpha (z) ,\\
\ & = (x<y+x>y+ x\cdot y)<\alpha (z)
+ (x<y+x>y+ x\cdot y)>\alpha (z)\\
\ & +(x<y+x>y+ x\cdot y)\cdot \alpha (z)\\
\ &= (x\ast y)\ast \alpha (z).
\end{align*}
\end{proof}

\section{Rota-Baxter operators and Hom-associative algebras } \label{sect3}
We extend in this section the notion of Rota-Baxter algebra to Hom-associatiative  algebras.
\begin{definition}
A  Hom-associative Rota-Baxter algebra is a Hom-associative algebra $( A, \cdot, \alpha) $  endowed with a linear map $R: A\rightarrow A$ subject to the relation
\begin{equation}\label{RBoperator}
R(x)\cdot R(y)=R(R(x)\cdot y+x\cdot R(y)+\theta x\cdot y),
\end{equation}
where $\theta\in\K$.

The map $R$ is called \emph{Rota-Baxter operator} of weight $\theta$ and the identity \eqref{RBoperator} \emph{Rota-Baxter identity}. We denote the  Hom-associative Rota-Baxter algebra by a quadruple $( A, \cdot, \alpha,R). $   We recover classical Rota-Baxter associative algebras when $\alpha = id$ and we denote them by triples $( A, \cdot,R). $
\end{definition}

\begin{remark}
Let $( A, \cdot, \alpha,R)$ be a Hom-associative Rota-Baxter algebra, where $R$ is a Rota-Baxter operator of weight $\theta$. Then  $( A, \cdot, \alpha,\theta id-R)$ is a Hom-associative Rota-Baxter algebra.
Indeed, the proof is straightforward and does not use the Hom-associativity of the algebra.
\end{remark}

In the following we  provide some constructions of Rota-Baxter Hom-algebras starting from classical Rota-Baxter algebra. Also we construct  new Rota-Baxter Hom-algebras from a given Rota-Baxter Hom-algebra. These constructions extend  to Rota-Baxter Hom-algebras the composition method, $n$th derived Hom-algebra construction and a construction involving elements of the centroid.

\begin{theorem}
\label{thm:GAssSALmorphism} Let $(A,\cdot, R)$ be an  associative Rota-Baxter algebra  and  $\alpha : A\rightarrow A$ be an algebra endomorphism commuting with $R$. Then
$(A,\cdot_\alpha,\alpha,R)$, where
$x\cdot_\alpha y=\alpha(x\cdot y)$,  is a  Hom-associative Rota-Baxter
algebra.
\end{theorem}
\begin{proof}The Hom-associative structure of the algebra follows from Yau's Theorem in \cite{Yau:HomolHomLie}.

 Now we check that $R$ is still a Rota-Baxter operator for the Hom-associative algebra.	
 \begin{eqnarray*}
R(x)\cdot _\alpha R(y)&=&\alpha (R(x)\cdot R(y)),\\
\ &=&\alpha(R(R(x)\cdot y+x\cdot R(y)+\theta x\cdot y)),\\
\ &=&\alpha(R(R(x)\cdot y))+\alpha(R(x\cdot R(y)))+\alpha(R(\theta x\cdot y))).\\
\end{eqnarray*}
Since $\alpha$ and $R$ commute then
\begin{eqnarray*}
R(x)\cdot _\alpha R(y)&=&R(\alpha(R(x)\cdot y))+R(\alpha(x\cdot R(y)))+R(\alpha(\theta x\cdot y))),\\
\ &=& R(R(x)\cdot_\alpha y+x\cdot_\alpha R(y)+\theta x_\alpha \cdot y)).\\
\end{eqnarray*}
\end{proof}



More generally, given a Hom-associative Rota-Baxter algebra $\left( A,\mu ,\alpha,R \right) $, one may ask whether this
Hom-associative Rota-Baxter algebra is   induced by an
ordinary associative Rota-Baxter  algebra $(A,\widetilde{\mu},R)$, that is $\alpha$
is an algebra endomorphism with respect to $\widetilde{\mu}$ and
$\mu=\alpha\circ\widetilde{\mu}$.

 Let $(A,\mu ,\alpha)$ be a multiplicative  Hom-associative algebra. It was observed in \cite{Gohr} that in case $\alpha$ is invertible, the composition method using $\alpha^{-1}$ leads to an associative algebra. If $\alpha$
is an algebra endomorphism with respect to $\widetilde{\mu}$
then $\alpha$ is also an algebra endomorphism
with respect to $\mu$. Indeed,
$$\mu(\alpha(x),\alpha(y))=\alpha\circ\widetilde{\mu}(\alpha(x),\alpha(y))=
\alpha\circ\alpha\circ\widetilde{\mu}(x,y)=\alpha\circ\mu(x,y).$$
\noindent
If $\alpha$ is bijective then $\alpha^{-1}$ is
also an algebra automorphism. Therefore one may use an untwist
operation on the Hom-associative algebra in order to recover the
associative algebra ($\widetilde{\mu}=\alpha^{-1}\circ\mu$).

\begin{proposition}
Let $(A,\mu ,\alpha,R)$ be a multiplicative  Hom-associative  Rota-Baxter algebra where  $\alpha$ is invertible  and such that $\alpha$ and $R$ commute.  Then  $(A,\mu_{\alpha^{-1}}=\alpha^{-1}\circ \mu ,R)$ is a Hom-associative  Rota-Baxter algebra.
\end{proposition}
\begin{proof}The associativity condition follows from
\begin{align*}
0&= \alpha^{-2}\mu(\alpha(x),\mu (y, z))-\mu(\mu(x, y), \alpha(z)),\\
\ &=  \alpha^{-1}\mu(x,\alpha^{-1} \mu(y, z))-\alpha^{-1}\mu(\alpha^{-1}\mu(x, y), z),\\
\ &=  \mu_{ \alpha^{-1}}(x,\mu_{ \alpha^{-1}}(y, z))-\mu_{ \alpha^{-1}}(\mu_{ \alpha^{-1}}(x, y), z).
\end{align*}
Since  $\alpha$ and $R$ commute then  $\alpha^{-1}$ and $R$ commute as well. Hence $R$ is a Rota-Baxter operator for the new multiplication.
\end{proof}

We may also derive new Hom-associative algebras from a given multiplicative Hom-associative algebra using the following procedure. We split the  definition given in  \cite{Yau4} into two types of   $n$th derived Hom-algebras.
\begin{definition}[\cite{Yau4}]\label{defiNDerivedHom} Let  $\left( A,\mu ,\alpha \right) $ be a multiplicative Hom-algebra and $n\geq 0$. The  $n$th derived Hom-algebra of type $1$ of $A$ is  defined by
\begin{equation}\label{DerivedHomAlgtype1}
A^n=\left( A,\mu^{(n)}=\alpha^{n }\circ\mu ,\alpha^{n+1} \right),
\end{equation}
and the $n$th derived Hom-algebra of type $2$ of $A$ is  defined by
\begin{equation}\label{DerivedHomAlgtype2}
A^n=\left( A,\mu^{(n)}=\alpha^{2^n -1}\circ\mu ,\alpha^{2^n} \right).
\end{equation}
Note that in both cases  $A^0=A$ and  $A^1=\left( A,\mu^{(1)}=\alpha\circ\mu ,\alpha^{2} \right)$.
\end{definition}
Observe that for $n\geq 1$ and $x,y,z\in A$ we have
\begin{eqnarray*}
\mu^{(n)}(\mu^{(n)}(x,y),\alpha^{n+1}(z))&=& \alpha^{n }\circ\mu(\alpha^{n }\circ\mu(x,y),\alpha^{n+1}(z))\\
\ &=& \alpha^{2n }\circ\mu(\mu(x,y),\alpha(z)).
\end{eqnarray*}
Therefore, following  \cite{Yau4}, one obtains the following result.
\begin{theorem}\label{ThmConstrNthDerivedAss}
Let   $\left( A,\mu ,\alpha,R \right) $ be a multiplicative Hom-associative Rota-Baxter algebra such that $\alpha $ and $R$ commute.Then the $n$th derived Hom-algebra of type $1$ is  also a Hom-associative Rota-Baxter algebra.
\end{theorem}

\begin{proof}
The operator $R$ is  a Rota-Baxter operator for the new multiplication since
$$
\alpha^n(\mu (R(x),R(y)))=\alpha^n(R(\mu (x, R(y)))+R(\mu (R(x), y))+\theta R(\mu (x, y))).
$$
\end{proof}

In the following we construct  Hom-associative Rota-Baxter algebras involving elements of the centroid of associative  Rota-Baxter algebras. The construction of Hom-algebras using elements of the centroid was initiated in \cite{BenayadiMakhlouf} for Lie algebras.

 Let   $(A, \cdot)$ be an associative algebra. An endomorphism  $\alpha\in End (A) $ is said to be an element of the centroid if $\alpha(x\cdot y)=\alpha(x)\cdot y=x \cdot \alpha (y)$  for any $x,y\in A$. The centroid of $A$ is defined by
$$Cent (A )=\{ \alpha\in End (A) : \alpha(x\cdot y)=\alpha(x)\cdot y=x \cdot \alpha (y), \  \forall x,y\in A\}.
$$
The same definition of the centroid  is assumed for Hom-associative algebras.
\begin{proposition}
Let   $(A, \mu,R)$ be an associative Rota-Baxter algebra where $R$ is a Rota-Baxter operator of weight $\theta$. Let   $\alpha\in Cent (A) $ and  set for $x,y\in A$
\begin{align*}
\mu_\alpha^1(x,y)=\mu (\alpha (x),y)\quad \text{and} \quad
\mu_\alpha^2(x,y) =\mu(\alpha (x),\alpha (y)).
\end{align*}
Assume that $\alpha$ and $R$ commute. Then $(A,\mu_\alpha^1,\alpha,R )$ and $(A,\mu_\alpha^2,\alpha ,R)$ are Hom-associative Rota-Baxter algebras.
\end{proposition}
\begin{proof}
Observe that
$$\mu_\alpha^1(\alpha (x),\mu_\alpha^1(y,z))=\mu(\alpha^2 (x),\mu(\alpha( y),z))=\mu(\alpha^2 (x),\alpha  \mu(y,z))=\alpha (\mu(\alpha (x),  \mu(y,z)))=\alpha^2 (\mu(x,  \mu(y,z))).
$$
Similarly
$$\mu_\alpha^2(\alpha (x),\mu_\alpha^2(y,z))=\mu(\alpha^2 (x),\alpha\mu(\alpha( y),\alpha (z)))=\mu(\alpha^2 (x),\alpha^2  \mu(y,z))=\alpha^2 (\mu(x,  \mu(y,z))).
$$

 The triple $(A,\mu_\alpha^1,\alpha)$ and $(A,\mu_\alpha^2,\alpha )$ are Hom-associative algebras.

They are  also  Rota-Baxter algebras since
\begin{align*}
\mu_\alpha^1(R(x), R(y))&=\mu (\alpha(R(x)), R(y))=\alpha(\mu(R(x), R(y))),\\
\ &=\alpha(R(\mu (x, R(y))+\mu (R(x), y)+\theta \mu (x, y))),\\
\ &=R(\mu (\alpha(x), R(y))+\mu (\alpha(R(x)), y)+\theta \mu (\alpha(x), y)),\\
\ &=R(\mu _\alpha^1(x, R(y))+\mu _\alpha^1(R(x), y)+\theta \mu _\alpha^1(x, y)).\\
\end{align*}
and
\begin{align*}
\mu_\alpha^2 (R(x), R(y))&=\mu( \alpha(R(x)), \alpha(R(y)))=\alpha(\mu( R(x), \alpha(R(y))))=-\alpha^2(\mu( R(y),R(x)))=\alpha^2(\mu( R(x),R(y))),\\
\ &=\alpha^2(R(\mu( x, R(y))+\mu( R(x), y)+\theta \mu( x, y))),\\
\ &=R(\alpha(\mu( \alpha(x), R(y)))+\alpha(\mu( \alpha(R(x)), y))+\theta \alpha(\mu( \alpha(x), y))),\\
\ &=-R(\alpha(\mu(  R(y),\alpha(x)))+\alpha(\mu(  y,\alpha(R(x))))+\theta \alpha(\mu(  y,\alpha(x)))),\\
\ &=-R(\mu(  \alpha(R(y)),\alpha(x))+\mu( \alpha( y),\alpha(R(x)))+\theta \mu(  \alpha(y),\alpha(x))),\\
\ &=R(\mu_\alpha^2(x, R(y)+\mu_\alpha^2( R(x), y)+\theta \mu_\alpha^2( x, y)).\\
\end{align*}
\end{proof}

\section{Some Functors}
We show in this section that there is a functor from the category of Hom-associative Rota-Baxter algebras to  the category of Hom-preLie algebras and then a functor  to the category of Hom-dendriform algebras and Hom-tridendriform algebras.

\begin{proposition}\label{HAssRBtoHpLie}
Let $( A, \cdot, \alpha,R) $ be a  Hom-associative Rota-Baxter algebra where $R$ is  a Rota-Baxter operator of weight $0$. Assume that $\alpha$ and $R$ commute. We define the operation $\ast$  on $A$ by
\begin{equation}
x\ast y=R(x)\cdot y-y\cdot R(x).
\end{equation}
Then $( A, \ast, \alpha) $ is a Hom-preLie algebra.
\end{proposition}
\begin{proof}
For $x,y\in A$ we have
\begin{align*}
\alpha (x)\ast (y\ast z)&= \alpha (x)\ast (R(y)\cdot z -z\cdot R(y)),\\
\ &= R(\alpha (x))\cdot (R(y)\cdot z -z\cdot R(y))-(R(y)\cdot z -z\cdot R(y))\cdot R (\alpha (x)),\\
\ &=R(\alpha (x))\cdot (R(y)\cdot z) -R(\alpha (x))\cdot (z\cdot R(y))-(R(y)\cdot z)\cdot R (\alpha (x)) +(z\cdot R(y))\cdot R (\alpha (x)),\\
\ &= \alpha (R(x))\cdot (R(y)\cdot z) -\alpha (R(x))\cdot (z\cdot R(y))-(R(y)\cdot z)\cdot \alpha (R (x)) +(z\cdot R(y))\cdot \alpha (R (x)),
\end{align*}
and
\begin{align*}
(x\ast y)\ast \alpha (z)&= (R(x)\cdot y-y\cdot R(x))\ast \alpha (z),\\
\ &=R(R(x)\cdot y-y\cdot R(x))\cdot \alpha (z)- \alpha (z)\cdot R(R(x)\cdot y-y\cdot R(x)),\\
\ &= R(R(x)\cdot y)\cdot \alpha (z)-(y\cdot R(x))\cdot \alpha (z)- \alpha (z)\cdot R(R(x)\cdot y)+\alpha (z)\cdot R(y \cdot R(x)).\\
\end{align*}
Then
\begin{eqnarray*}
&&\alpha (x)\ast (y\ast z)-(x\ast y)\ast \alpha (z) -\alpha (y)\ast (x\ast z)+(y\ast x)\ast \alpha (z)=\\
&& \alpha (R(x))\cdot (R(y)\cdot z) -\alpha (R(x))\cdot (z\cdot R(y))-(R(y)\cdot z)\cdot \alpha (R (x)) +(z\cdot R(y))\cdot \alpha (R (x))\\
&&-R(R(x)\cdot y)\cdot \alpha (z)+R(y\cdot R(x))\cdot \alpha (z)+ \alpha (z)\cdot R(R(x)\cdot y)-\alpha (z)\cdot R(y \cdot R(x))\\
&& -\alpha (R(y))\cdot (R(x)\cdot z) +\alpha (R(y))\cdot (z\cdot R(x))+(R(x)\cdot z)\cdot \alpha (R (y)) -(z\cdot R(x))\cdot \alpha (R (y))\\
&&+R(R(y)\cdot x)\cdot \alpha (z)-R(x\cdot R(y))\cdot \alpha (z)-\alpha (z)\cdot R(R(y)\cdot x)+\alpha (z)\cdot R(x \cdot R(y)).
\end{eqnarray*}
Using the Rota-Baxter identity \ref{RBoperator} we gather the $5^{th}$ and $14^{th}$, $6^{th}$ and $13^{th}$, $7^{th}$ and $16^{th}$, $8^{th}$ and $15^{th}$ terms. Therefore we obtain
\begin{eqnarray*}
&&\alpha (x)\ast (y\ast z)-(x\ast y)\ast \alpha (z) -\alpha (y)\ast (x\ast z)+(y\ast x)\ast \alpha (z)=\\
&& \alpha (R(x))\cdot (R(y)\cdot z) -\alpha (R(x))\cdot (z\cdot R(y))-(R(y)\cdot z)\cdot \alpha (R (x)) +(z\cdot R(y))\cdot \alpha (R (x))\\
&&-(R(x)\cdot R(y))\cdot \alpha (z)+(R(y)\cdot R(x))\cdot \alpha (z)+ \alpha (z)\cdot (R(x)\cdot R(y))-\alpha (z)\cdot (R(y) \cdot R(x))\\
&& -\alpha (R(y))\cdot (R(x)\cdot z) +\alpha (R(y))\cdot (z\cdot R(x))+(R(x)\cdot z)\cdot \alpha (R (y)) -(z\cdot R(x))\cdot \alpha (R (y)).
\end{eqnarray*}
Then Hom-associativity leads to
\begin{eqnarray*}
\alpha (x)\ast (y\ast z)-(x\ast y)\ast \alpha (z) -\alpha (y)\ast (x\ast z)+(y\ast x)\ast \alpha (z)=0.
\end{eqnarray*}
\end{proof}

\begin{proposition}
Let $( A, \cdot, \alpha,R) $ be a  Hom-associative Rota-Baxter algebra where $R$ is  a Rota-Baxter operator of weight $-1$. Assume that $\alpha$ and $R$ commute. We define the operation $\ast$  on $A$ by
\begin{equation}
x\ast y=R(x)\cdot y-y\cdot R(x)-x\cdot y.
\end{equation}
Then $( A, \ast, \alpha) $ is a Hom-preLie algebra.
\end{proposition}
\begin{proof}
For $x,y\in A$ we have
\begin{align*}
\alpha (x)\ast (y\ast z)&= R(\alpha (x))\cdot (R(y)\cdot z -z\cdot R(y)-y\cdot z) -(R(y)\cdot z -z\cdot R(y)-y\cdot z)\cdot R (\alpha (x)),\\
\ &-\alpha (x)\cdot (R(y)\cdot z-z\cdot R(y)-y\cdot z),
\end{align*}
and
\begin{align*}
(x\ast y)\ast \alpha (z)&= R(R(x)\cdot y-y\cdot R(x)-x\cdot y)\cdot \alpha (z) - \alpha (z)\cdot R(R(x)\cdot y-y\cdot R(x)-x\cdot y ),\\
\ &-(R(x)\cdot y-y\cdot R(x)-x\cdot y)\cdot \alpha (z).
\end{align*}
Then using the fact that $\alpha$ and $R$ commute, and the Hom-associativity we obtain
\begin{eqnarray*}
&&\alpha (x)\ast (y\ast z)-(x\ast y)\ast \alpha (z) -\alpha (y)\ast (x\ast z)+(y\ast x)\ast \alpha (z)=\\
&& \alpha (R(x))\cdot (R(y)\cdot z) +(z\cdot R(y))\cdot \alpha (R (x))-R(R(x)\cdot y)\cdot \alpha (z)+R(y\cdot R(x))\cdot \alpha (z)\\
&&+ \alpha (z)\cdot R(R(x)\cdot y)-\alpha (z)\cdot R(y \cdot R(x)) -\alpha (R(y))\cdot (R(x)\cdot z) -(z\cdot R(x))\cdot \alpha (R (y))\\
&&+R(R(y)\cdot x)\cdot \alpha (z)-R(x\cdot R(y))\cdot \alpha (z)-\alpha (z)\cdot R(R(y)\cdot x)+\alpha (z)\cdot R(x \cdot R(y)).
\end{eqnarray*}
Then it vanishes using the Rota-Baxter identity \ref{RBoperator}.
\end{proof}

Now we connect Hom-associative Rota-Baxter  algebras to Hom-dendriform algebras. We generalize to Hom-algebras setting, the result given by Aguiar for weight $0$ Rota-Baxter algebras in \cite{Aguiar} and extended by Ebrahimi-Fard in \cite{KEF1} to any Rota-Baxter algebras.
\begin{proposition}\label{prop4.3}
Let $( A,\cdot, \alpha,R) $ be a  Hom-associative Rota-Baxter algebra where $R$ is a Rota-Baxter operator of weight $0$. Assume that $\alpha$ and $R$ commute. We define the operation $\prec$ and $\succ$ on $A$ by
\begin{equation}
x\prec y=x\cdot R(y)\quad\text{and}\quad x\succ y=R(x)\cdot y.
\end{equation}
Then $( A,\prec, \succ, \alpha) $ is a Hom-dendriform algebra.
\end{proposition}
\begin{proof}Let $x,y,z\in A$, we have by using Hom-associativity, identity \eqref{RBoperator} and the fact that  $\alpha$ and $R$ commute:
\begin{eqnarray*} (x\prec y)\prec \alpha (z)-\alpha(x)\prec(y\prec z +y\succ z)&=&
(x\cdot R(y))\cdot R(\alpha (z))-\alpha (x)\cdot R( y\cdot R(z)+R(y)\cdot z),\\
\ &=&(x\cdot R(y))\cdot \alpha (R(z))-\alpha (R(x))\cdot( R(y)\cdot R(z)),\\ \ &=&0
\end{eqnarray*}
\begin{eqnarray*}
 (x\succ y)\prec \alpha (z)-\alpha(x)\succ(y\prec z)&=&(R(x)\cdot y)\cdot R(\alpha (z))-R(\alpha (x))\cdot( y\cdot R(z))\\
\ &=&(R(x)\cdot y)\cdot \alpha (R(z))-\alpha (R(x))\cdot( y\cdot R(z)),\\ \ &=&0.
\end{eqnarray*}
\begin{eqnarray*} \alpha(x)\succ(y\succ z)-(x\prec y+x\succ y)\succ \alpha (z)&=&R(\alpha(x))\cdot(R(y)\cdot z)-R(x\cdot R(y)+R(x)\cdot y)\cdot \alpha (z),\\
\ &=&\alpha(R(x))\cdot(R(y)\cdot z)-(R(x)\cdot R(y))\cdot \alpha (z),\\
\ &=&0.
\end{eqnarray*}

\end{proof}

\begin{proposition}\label{pro4.4}
Let $( A,\cdot, \alpha,R) $ be a  Hom-associative Rota-Baxter algebra where $R$ is  a Rota-Baxter operator of weight $\theta$. Assume that $\alpha$ and $R$ commute. We define the operation $\prec$, $\succ$ and $\bullet$ on $A$ by
\begin{equation}
x\prec y=x\cdot R(y)+\theta x\cdot y,\quad\text{and}\quad x\succ y=R(x)\cdot y.
\end{equation}
Then $( A,\prec, \succ, \bullet,\alpha) $ is a Hom-dendriform algebra.
\end{proposition}
\begin{proof}A straightforward calculation permits to check the axioms. For example
\begin{align*}
 (x\succ y)\prec \alpha (z)-\alpha(x)\succ(y\prec z)&=& &(R(x)\cdot y)\prec \alpha (z)-\alpha (x)\succ( y\cdot R(z)+\theta y\cdot z),\\
\ &=& &(R(x)\cdot y)\cdot R(\alpha (z))+\theta(R(x)\cdot y)\cdot \alpha (z)-R(\alpha (x))\cdot( y\cdot R(z)+\theta y\cdot z),\\
\ &=& & 0.
\end{align*}
\end{proof}
\begin{remark}
Proposition \ref{HAssRBtoHpLie} could be obtained as a corollary of Proposition \ref{prop4.3}  and  Proposition \ref{HDendtoHpLie} which leads  to a construction of right Hom-preLie algebra with the following multiplication
$$x\lhd y=x\cdot R(y)- R(y)\cdot x.
$$
\end{remark}

Considering the associated categories and denoting by $HRBass_\theta$ the category of Hom-associative Rota-Baxter algebras, $HpreLie$ the category of preLie algebras and $Hdend$ the category of Hom-dendriform algebras, we summarize the previous results in  the following proposition

\begin{proposition}
The following diagram is commutative
\[
\begin{array}{ccc}
  Hdend&  \rightarrow & HpreLie  \\
 \downarrow & \  &    \downarrow \\
HRBass_0  & \rightarrow  & HpreLie
\end{array}
\]

\end{proposition}

We show now a connection between  Hom-associative Rota-Baxter  algebras and Hom-tridendriform algebras. The classical case was stated in  \cite{KEF1}.
\begin{proposition}
Let $( A,\cdot, \alpha,R) $ be a  Hom-associative Rota-Baxter algebra where $R$ is  a Rota-Baxter operator of weight $\theta$. Assume that $\alpha$ and $R$ commute. We define the operation $\prec$, $\succ$ and $\bullet$ on $A$ by
\begin{equation}
x\prec y=x\cdot R(y),\quad\quad x\succ y=R(x)\cdot y\quad\text{and}\quad x\bullet y=\theta x\cdot y.
\end{equation}
Then $( A,\prec, \succ, \bullet,\alpha) $ is a Hom-tridendriform algebra.
\end{proposition}
\begin{proof}
The first three axioms follow from Proposition \ref{pro4.4} and the last use Hom-associativity and the commutation between $R$ and $\alpha$. For example
\begin{align*}
 (x\prec y)\bullet \alpha (z)-\alpha(x)\bullet(y\succ z)=\theta (x\cdot R( y))\cdot \alpha (z)-\theta\alpha (x)\cdot( R(y)\cdot  z)=0.
\end{align*}

\end{proof}
Following Proposition \ref{HTridendToHAss}, we derive a new Hom-associative multiplication defined by
$$x\ast y=x\cdot R(y)+R(x)\cdot y+\theta x\cdot y.
$$

As in the classical case it satisfies
$$R(x\ast y)=R(x)\cdot R(y)\quad\text{and}\quad \widetilde{R}(x\ast y)=-\widetilde{R}(x)\cdot \widetilde{R}(y)
$$
where $\widetilde{R}(x)=-\theta x-R(x).$ \\

\section{Rota-Baxter operators and Hom-Nonassociative algebras }
Rota-Baxter operator  in the context of Lie algebras were introduced
independently by Belavin and Drinfeld and Semenov-Tian-Shansky \cite{Drinfeld,semenov} in the 1980th
and were related to solutions of the (modified) classical Yang-Baxter
equation. The theory were developed later  by Ebrahimi-Fard see \cite{KEF1}.

We may extend the theory of Hom-associative Rota-Baxter algebras developed above to any Hom-Nonassociative algebra. We set the following definition
\begin{definition}
A  Hom-Nonassociative Rota-Baxter algebra is a Hom-Nonassociative algebra $( A, [\ ,\ ], \alpha) $  endowed with a linear map $R: A\rightarrow A$ subject to the relation
\begin{equation}\label{RBoperatorNonAss}
[R(x), R(y)]=R([R(x), y]+[x, R(y)]+ \theta [x,y]),
\end{equation}
where $\theta\in\K$.

The map $R$ is called \emph{Rota-Baxter operator} of weight $\theta$.
\end{definition}

We obtain the following construction by composition of Hom-Lie Rota-Baxter algebras, extending the construction of Hom-Lie algebras given by Yau in \cite{Yau:homology} to Rota-Baxter algebras.
\begin{theorem}\label{thmYauConstrHomLie}
Let $(\g,[~,~],R)$ be a Lie Rota-Baxter algebra and $\alpha :
\g\rightarrow \g$ be a Lie algebra endomorphism commuting with $R$. Then
$(\g,[~,~]_\alpha,\alpha,R)$, where $[~,~]_\alpha=\alpha\circ[~,~]$,  is a Hom-Lie Rota-Baxter algebra.
\end{theorem}

\begin{proof}
Observe that $[\alpha (x),[y,z]_\alpha]_\alpha =\alpha [\alpha
(x),\alpha[y,z]]=\alpha^2 [x,[y,z]] $. Therefore the Hom-Jacobi
identity for $\mathfrak{g}_\alpha=(\g,[~,~]_\alpha,\alpha)$ follows
obviously from the Jacobi identity of $(\g,[~,~])$. The
skew-symmetry is proved similarly.

Now we check that $R$ is still a Rota-Baxter operator for the Hom-Lie  algebra.	
 \begin{eqnarray*}
[R(x),  R(y)]_\alpha &=&\alpha [R(x), R(y)],\\
\ &=&\alpha(R([R(x), y]+[x, R(y)]+\theta [x, y])),\\
\ &=&\alpha(R([R(x), y]))+\alpha(R([x, R(y)]))+\alpha(R(\theta [x, y])).\\
\end{eqnarray*}
Since $\alpha$ and $R$ commute then
\begin{eqnarray*}
[R(x),  R(y)]_\alpha&=&R(\alpha([R(x), y]))+R(\alpha([x, R(y)]))+R(\alpha(\theta [x, y])),\\
\ &=& R([R(x), y]_\alpha+[x,R(y)]_\alpha+\theta [x, y]_\alpha).\\
\end{eqnarray*}
\end{proof}

\begin{remark}
In particular  the proposition is valid  when $\alpha$ is an involution.
\end{remark}

Let $(\g,[~,~],\alpha)$ be a  Hom-Lie algebra. It was observed in \cite{Gohr} that in case $\alpha$ is invertible, the composition method using $\alpha^{-1}$ leads to a Lie algebra.

\begin{proposition}\label{CompoMethodLie}
Let $(\g,[\ ,\ ] ,\alpha,R)$ be a   Hom-Lie Rota-Baxter algebra such that $\alpha$ and $R$ commute.  Then  $(\g,[~,~]_{\alpha^{-1}}=\alpha^{-1}\circ [~,~],R)$ is a Lie Rota-Baxter algebra.
\end{proposition}
\begin{proof}The Jacobi identity follows from
\begin{align*}
\circlearrowleft_{x,y,z}{[x,[y ,z]_{\alpha^{-1}}]_{\alpha^{-1}}}=\circlearrowleft_{x,y,z}{\alpha^{-1}( [x,\alpha^{-1} ([y,z] ) ])}=\circlearrowleft_{x,y,z}{{\alpha^{-2}}[\alpha (x),[y,z]]}=0.
\end{align*}
Since  $\alpha$ and $R$ commute then  $\alpha^{-1}$ and $R$ commute as well. Hence $R$ is a Rota-Baxter operator for the new multiplication.
\end{proof}

We  construct now  new Hom-Lie Rota Baxter algebras  from a given multiplicative  Hom-Lie Rota-Baxter  algebra   using $n$th derived Hom-Lie algebras.

\begin{definition}[\cite{Yau4}]\label{defiNDerivedHom} Let  $\left( \g,[\ ,\ ] ,\alpha \right) $ be a multiplicative Hom-Lie algebra and $n\geq 0$. The  $n$th derived Hom-algebra  of $\g$ is  defined by
\begin{equation}\label{DerivedHomAlgtype1}
\g _{(n)}=\left( \g,[\ ,\ ]^{(n)}=\alpha^{n }\circ[\ ,\ ] ,\alpha^{n+1} \right),
\end{equation}
Note that $\g_{(0)}=\g$ and  $\g_{(1)}=\left( \g,[\ ,\ ]^{(1)}=\alpha\circ[\ ,\ ] ,\alpha^{2} \right)$.
\end{definition}
Observe that for $n\geq 1$ and $x,y,z\in \g$ we have
\begin{eqnarray*}
[[x,y]^{(n)},\alpha^{n+1}(z)]^{(n)}&=& \alpha^{n }([\alpha^{n }([x,y]),\alpha^{n+1}(z)])\\
\ &=& \alpha^{2n }([[x,y],\alpha(z)]).
\end{eqnarray*}

\begin{theorem}\label{ThmConstrNthDerivedLie}
Let   $\left( \g,[\ ,\ ] ,\alpha,R \right) $ be a multiplicative Hom-Lie Rota-Baxter algebra and assume that $\alpha$ and $R$ commute. Then its  $n$th derived Hom-algebra  is  a Hom-Lie Rota-Baxter algebra.
\end{theorem}
\begin{proof}
The $n$th derived Hom-algebra is a Hom-Lie algebra according to \cite{Yau4}.

It is also a Rota-Baxter algebra since
$$
\alpha^n([R(x),R(y)])=\alpha^n(R([x, R(y)])+R([R(x), y])+\lambda R([x, y])).
$$
\end{proof}

In the following we construct  Hom-Lie Rota-Baxter algebras involving elements of the centroid of Lie Rota-Baxter algebras.  Let   $(\mathfrak{g}, [\cdot, \cdot],R)$ be a Lie Rota-Baxter algebra. An endomorphism  $\alpha\in End (\g) $ is said to be an element of the centroid if $\alpha[x,y]=[\alpha(x),y]$  for any $x,y\in\g$. The centroid is defined by
$$Cent (\g )=\{ \alpha\in End (\g) : \alpha [x,y]=[\alpha (x),y], \  \forall x,y\in\g\}.
$$
The same definition of the centroid  is assumed for Hom-Lie Rota-Baxter algebra.
\begin{proposition}
Let   $(\mathfrak{g}, [\cdot, \cdot],R)$ be a Lie Rota-Baxter algebra where $R$ is a Rota-Baxter operator of weight $\theta$. Let   $\alpha\in Cent (\g) $ and  set for $x,y\in\g$
\begin{align*}
[x,y]_\alpha^1=[\alpha (x),y]\quad \text{and} \quad
[x,y]_\alpha^2 =[\alpha (x),\alpha (y)].
\end{align*}
Assume that $\alpha$ and $R$ commute. Then $(\mathfrak{g},[\cdot, \cdot ]_\alpha^1,\alpha,R )$ and $(\mathfrak{g},[\cdot, \cdot]_\alpha^2,\alpha ,R)$ are Hom-Lie Rota-Baxter algebras.
\end{proposition}
\begin{proof}
 The triple $(\mathfrak{g},[\cdot, \cdot ]_\alpha^1,\alpha,R )$ and $(\mathfrak{g},[\cdot, \cdot]_\alpha^2,\alpha ,R)$ are Hom-Lie algebras  according to \cite[Proposition  1.12]{BenayadiMakhlouf}.

They are  also  Rota-Baxter algebras since
\begin{align*}
[R(x), R(y)]_\alpha^1&=[\alpha(R(x)), R(y)]=\alpha([R(x), R(y)]),\\
\ &=\alpha(R([x, R(y)]+[R(x), y]+\theta [x, y])),\\
\ &=R([\alpha(x), R(y)]+[\alpha(R(x)), y]+\theta [\alpha(x), y]),\\
\ &=R([x, R(y)]_\alpha^1+[R(x), y]_\alpha^1+\theta [x, y]_\alpha^1).\\
\end{align*}
and
\begin{align*}
[R(x), R(y)]_\alpha^2&=[\alpha(R(x)), \alpha(R(y))]=\alpha([R(x), \alpha(R(y))])=-\alpha^2([R(y),R(x)])=\alpha^2([R(x),R(y)]),\\
\ &=\alpha^2(R([x, R(y)]+[R(x), y]+\theta [x, y])),\\
\ &=R(\alpha([\alpha(x), R(y)])+\alpha([\alpha(R(x)), y])+\theta \alpha([\alpha(x), y])),\\
\ &=-R(\alpha([ R(y),\alpha(x)])+\alpha([ y,\alpha(R(x))])+\theta \alpha([ y,\alpha(x)])),\\
\ &=-R([ \alpha(R(y)),\alpha(x)]+[\alpha( y),\alpha(R(x))]+\theta [ \alpha(y),\alpha(x)]),\\
\ &=R([x, R(y)]_\alpha^2+[R(x), y]_\alpha^2+\theta [x, y]_\alpha^2).\\
\end{align*}
\end{proof}


We may obtain similar connections to  Hom-preLie and Hom-dendriform algebras as for Hom-associative algebras for example we have
\begin{proposition}
Let $( A,[\ ,\ ], \alpha,R) $ be a  Hom-Lie Rota-Baxter algebra where $R$ is  a Rota-Baxter operator  of weight $0$. Assume that $\alpha$ and $R$ commute. We define the operation $\ast$  on $A$ by
\begin{equation}
x\ast y=[R(x), y].
\end{equation}
Then $( A, \ast, \alpha) $ is a Hom-preLie algebra.
\end{proposition}

\begin{remark}
The connection between  Rota-Baxter Hom-algebras and Yang-Baxter equation  will be developed in a forthcoming paper, as well as free Hom-associative Rota-Baxter algebra.
\end{remark}

\bibliographystyle{amsplain}
\providecommand{\bysame}{\leavevmode\hbox to3em{\hrulefill}\thinspace}
\providecommand{\MR}{\relax\ifhmode\unskip\space\fi MR }
\providecommand{\MRhref}[2]{%
  \href{http://www.ams.org/mathscinet-getitem?mr=#1}{#2}
}
\providecommand{\href}[2]{#2}

\end{document}